\UseRawInputEncoding
\documentclass[12pt,reqno]{amsart}
\usepackage{amsmath,amssym b,mathrsfs}
\usepackage{graphicx,cite,times}

\usepackage{cite}
\usepackage{appendix}
\usepackage{enumerate}

\newtheorem{theo}{Theorem}[section]

\newtheorem{lemm}[theo]{Lemma}

\newtheorem{defi}[theo]{Definition}
\newtheorem{rema}[theo]{Remark}

\numberwithin{equation}{section}

\begin{document}

\title{linearized inverse problem for biharmonic operators at high frequencies}
\author{Xiaomeng Zhao and Ganghua Yuan}
\address{KLAS, School of Mathematics and Statistics, Northeast Normal University,
Changchun, Jilin, 130024, China}
\email{zhaoxm600@nenu.edu.cn}
\email{yuangh925@nenu.edu.cn}
\thanks{
The second author is supported by NSFC grants 11771074 and National Key R\&D Program of China (No. 2020YFA0714102).}


\begin{abstract}
In this paper, we study the phenomenon of increasing stability in the inverse boundary value problems for the biharmonic equation. By considering a linearized form, we obtain an increasing Lipschitz-like stability when $k$ is large. Furthermore, we extend the discussion to the linearized inverse biharmonic potential problem with attenuation, where an exponential dependence of the attenuation constant is traced in the stability estimate.

\end{abstract}

\maketitle

\section{introduction}
In this paper, we consider the boundary-value problem for a perturbation of the biharmonic operator posed in a bounded domain $\Omega\subset\mathbb{R}^n$, $n\geq3$  with smooth boundary, equipped with the Navier boundary conditions, that is,
\begin{align}
\left\{
\begin{aligned}
	(\Delta^2-k^4+q)u&=0\quad \mbox{in }\Omega,\\
	u&=f\quad \mbox{on }\partial\Omega,\\
	\Delta u&=g \quad \mbox{on }\partial\Omega,
\end{aligned}
\right.
\label{m1}\end{align}
where $f\in H^{7/2}(\partial\Omega)$ and $g\in H^{3/2}(\partial\Omega)$. Biharmonic operators (with potentials) are widely studied in the context of modelling of hinged elastic beams and suspension bridges. We would like to refer to \cite{gazzola2010polyharmonic} for a discussion of these models and other applications.

If $k$ is not an eigenvalue of $(\Delta^2-k^4+q)u=0$ on the set $E=\{u\in H^4(\Omega):u|_{\partial\Omega}=\Delta u|_{\partial\Omega}=0\}$, there exists a unique solution to the problem (\ref{m1}) when $(f,g)\in H^{\frac{7}{2}}(\partial\Omega)\times H^{\frac{3}{2}}(\partial\Omega)$.

Let us define the set  $$\mathcal{Q}:=\{r\in L^\infty(\Omega):0 \ \mbox{is not an eigenvalue of }(\Delta^2-r) \ \mbox{on } E\}.$$

The Dirichlet-to-Neumann(DtN) map $\mathcal{N}_q$ can be defined as
\begin{align}
\begin{split}
\mathcal{N}_q:H^{\frac{7}{2}}(\partial\Omega)\times H^{\frac{3}{2}}(\partial\Omega)&\rightarrow H^{\frac{5}{2}}(\partial\Omega)\times H^{\frac{1}{2}}(\partial\Omega),\\
(f,g)&\mapsto \left(\frac{\partial u}{\partial \nu}|_{\partial \Omega},\frac{\partial(\Delta u)}{\partial \nu}|_{\partial \Omega}\right),
\end{split}\label{m1.2}
\end{align}
where $u\in H^4(\Omega)$ is the unique solution to (\ref{m1}).

On the space $H^\alpha(\partial\Omega)\times H^\beta(\partial\Omega)$ (for simplicity we will denote this space by $H^{\alpha,\beta}(\partial\Omega)$), we shall consider the norm
$$\|(f,g)\|_{H^{\alpha,\beta}(\partial\Omega)}:=\|f\|_{H^\alpha(\partial\Omega)}+\|g\|_{H^\beta(\partial\Omega)}.$$
Define: $$\|\mathcal{N}_q\|:=\mbox{sup}\{\|\mathcal{N}_q(f,g)\|_{H^{\frac{5}{2},\frac{1}{2}}(\partial\Omega)}:\|(f,g)\|_{H^{\frac{7}{2},\frac{3}{2}}(\partial\Omega)}=1\},$$
where $\mathcal{N}_q(f,g)$ is defined in (\ref{m1.2}).

Our aim, in this article, is to address the question of stability in inverse boundary value problems for perturbed biharmonic operators. That is, to determine the potential $q$ from DtN map. In recent years, this kind of inverse problems have been studied by researchers and some related results have been obtained. For example, stability estimates of logarithmic type for this problem were obtained in  \cite{choudhury2017stability} and \cite{choudhury2015stability} when $k=0$.

 According to the work \cite{mandache2001exponential}, logarithmic stability estimates are expected to be optimal for such inverse boundary value problems when the operator $\Delta^2$ is replaced by $\Delta$ and $k=0$.  See also \cite{ssylbekov2017determining},
 \cite{assylbekov2017determining},
\cite{bhattacharyya2019inverse}, \cite{choudhury2015stability},
\cite{ghosh2016determination},
\cite{ikehata1991special},
\cite{isakov1991completeness},
\cite{krupchyk2016inverse},
\cite{krupchyk2012determining}, \cite{krupchyk2014inverse},
\cite{serov2016borg} for this kind of inverse boundary value problems for perturbed biharmonic and polyharmonic operators.
The logarithmic stability estimates indicate that the inverse boundary value problem in question is severely ill-posed, which makes it most challenging to design reconstruction algorithms with high resolution in practice, since small errors in measurements may result in exponentially large errors in the reconstruction of the unknown potential.

However, in recent years, it has been widely observed in different inverse boundary value problems
that the stability estimate improves with growing wavenumbers (or energy) both
analytically and numerically. The increasing stability in the inverse boundary value
problem for the Schr\"odinger equation was firstly observed within different ranges
of the wavenumbers in \cite{isakov2011increasing}. For related stability estimates at high frequencies  for several fundamental inverse boundary value problems, we can refer to  \cite{hrycak2004increased},
\cite{isakov2007increased},  \cite{isakov2016increasing}, \cite{iixsakov2014increasing},   \cite{isakov2013increasing}, \cite{krupchyk2019stability}, \cite{liang2015increasing},
 \cite{liu2019stability} for example. Among those, in \cite{liu2019stability}, the author obtained an increasing H\"older-type stability for the biharmonic operator when $k$ is large.

Very recently, a linearized form of the inverse Schr\"odinger potential
problem  with a large wavenumber is considered in \cite{isakovlu2020linearized}, \cite{zou2022linearized}, \cite{isakovluxu2022linearized}. Such a
linearization is carried out at a zero potential function which is reasonable when the potential is sufficiently
small compared with the squared wavenumber $k^2$. Noting that the squared
wavenumber $k^2$ may be large, it is reasonable to reconstruct a potential function of
moderate amplitude. It shows that this kind of  linearization method can improve stability dramatically for a larger
$k$ (see e.g. \cite{isakovlu2020linearized}, \cite{zou2022linearized}, \cite{isakovluxu2022linearized}). And the linearization method can be adapted to deal with inverse boundary problems for some nonlinear Schr\"odinger operators (see e.g. \cite{lu2022increasing}).

Another topic on increasing stability is multi-frequency set-up uses finite observations for multiple wavenumbers.
To have an overview of such results, we recommend a recent review paper [3] which
nicely summarizes the theoretical and numerical evidences verifying the increasing
stability in inverse medium and source problems for acoustic Helmholtz equations
and Maxwell equations. For this topic, we can also see e.g.  \cite{bao2020stability}, \cite{chengIslu2016stability}, \cite{isakovlu2018multi}, \cite{li2016increasing}, \cite{li2017stability},  \cite{li2021stability}, \cite{li2021stability2}, \cite{li2020stability}.

Inspired by the formulation of the problem in \cite{isakovlu2020linearized}, in this paper, we investigate a linearized form of the inverse  boundary problem for (\ref{m1}) with a large wavenumber in three and higher dimensions.
 And we obtain an increasing Lipschitz-type stability estimate for the linearized inverse boundary problem for the biharmonic equation with a large wavenumber. Furthermore, we provide some extended discussion on the linearized inverse biharmonic potential problem with attenuation and we  obtain an increasing stability with the explicit dependance of  the attenuation constant.

The remainder of the paper is organized as follows.  In section 2, we introduce the linearized problem and prove an increasing Lipschitz-like stability estimate for the linearized inverse biharmonic potential problem with a large wavenumber by using bounded complex exponential solutions. We extend, in section 3, the discussion of the inverse biharmonic potential problem with attenuation. Finally a conclusion section 4 ends the manuscript with further prospects.

\section{Increasing stability in the Linearized inverse biharmonic potential problem at the large wavenumber}
In this section, we provide notations and statement of the main result for the Linearized inverse biharmonic potential problem. Some extended discussions on the linearized DtN map will also be provided.

\subsection{Linearized inverse biharmonic potential problem and main result}
We recall the original problem, which initializes the current work, is to find the potential function $q=q(x)$ defined in a bounded domain $\Omega$ in the following problem
\begin{align}
\left\{
\begin{aligned}
	(\Delta^2-k^4+q)u&=0\quad \mbox{in }\Omega,\\
	u&=f\quad \mbox{on }\partial\Omega,\\
	\Delta u&=g \quad \mbox{on }\partial\Omega,
\end{aligned}
\right.
\label{f}\end{align}
from the knowledge of the DtN map
$$\mathcal{N}_q :(f,g)\longmapsto(\partial_\nu u,\partial_\nu (\Delta u)) \quad \mbox{on }\partial\Omega.$$

If we assume that $q$ is small (or $k$ is large), we can justify the linearization of the biharmonic equation. More precisely, let $u_0$, $u_1$ solve the following subproblems
\begin{align}
\left\{
\begin{aligned}
\Delta^2u_0-k^4u_0&=0\quad \mbox{in }\Omega,\\
	u_0&=f\quad \mbox{on }\partial\Omega,\\
	\Delta u_0&=g \quad\mbox{on }\partial\Omega,
\end{aligned}
\right.
\label{m2}\end{align}
\begin{align}
\left\{
\begin{aligned}
	\Delta^2u_1-k^4u_1&=-qu_0&&\mbox{in } \Omega,\\
	u_1&=0 &&\mbox{on }\partial\Omega,\\
	\Delta u_1&=0 &&\mbox{on }\partial\Omega,
\end{aligned}
\right.
\label{m3}\end{align}
then the solution $u$ of the original problem (\ref{m1}) is
$$u=u_0+u_1+\cdots$$where the remaining ``$\cdots$" are ``higher" order terms. The linearized DtN map $\mathcal{N}_q^1$ of $\mathcal{N}_q$ in (\ref{m1.2}) is defined accordingly as
\begin{align}
	\mathcal{N}_q^1:(f,g)\longmapsto(\partial_\nu u_1,\partial_\nu (\Delta u_1)) \quad \mbox{on }\partial\Omega.
	\label{m4}
\end{align}

We consider the stability question for the determination of $q$ from the  linearized DtN map $\mathcal{N}_q^1$ defined by (\ref{m4}).

To begin with, let us recall the Green's formula
\begin{align*}
	\int_\Omega(\Delta^2u)v\,dx-\int_\Omega u(\Delta^2v)\,dx=&\int_{\partial\Omega}\partial_\nu(\Delta u)v\,dS+\int_{\partial\Omega}\partial_\nu u(\Delta v)\,dS\\
&-\int_{\partial\Omega}(\Delta u)\partial_\nu v\,dS-\int_{\partial\Omega}u(\partial_\nu(\Delta v))\,dS
\end{align*}
for $u,v\in H^4(\Omega)$.

Let  $v\in H^4(\Omega)$ be the solution to$$\Delta^2v-k^4v=0\quad\mbox{in }\Omega,$$
multiplying both sides of the subproblem (\ref{m3}) with by $v$, using the Green's formula, we obtain
\begin{align}
	\int_{\partial\Omega}\partial_\nu(\Delta u_1)v\,dS+\int_{\partial\Omega}\partial_\nu u_1(\Delta v)\,dS=-\int_\Omega qu_0v\,dx
\label{e1}\end{align}
We now state the main stability estimate below.
\begin{theo}
Let $n\geq 3$, $s>n/2$. Given any wavenumber $k\geq 1$  such that $k^4\in\mathcal{Q}$, assume the potential function $q\in H^s(\Omega)$ with supp$q\subset\Omega, \| q\|_{H^{s}(\Omega)}\leq M$. Then the following estimate holds true
\begin{align}\|q\|_{H^{-s}(\Omega)}\leq Ck^7\|\mathcal{N}_q^1\|+C(k+\log\frac{1}{\|\mathcal{N}_q^1\|^2})^{-(2s-n)}
\label{mm}\end{align}	
for the linearized system (\ref{m2})-(\ref{m3}) with  $\|\mathcal{N}_q^1\|\leq 1/e$ and the constant $C$ depends only on $n,s,\Omega, M$ and \mbox{supp}q.

\label{th1}	
\end{theo}
\begin{rema}
From estimate (\ref{mm}), it is obvious that the stability behaves more like Lipschitz type when $k$ is large. So we call it increasing Lipschitz-like stability in this paper.

\end{rema}
\subsection{Extended discussion on the linearized DtN map}
In this subsection, we analyze the linearization estimate and state the relation between the nonlinear operator $\mathcal{N}_q$ and the linearized DtN operator $\mathcal{N}_q^1$ if the potential function q is assumed to be sufficiently small. To this end, we define the following two linear operators.

\begin{defi}For all $q \in \mathcal{Q} \subset L^\infty (\Omega)$, we define the Poisson operator $\mathcal{P}_q:H^{\frac{7}{2}}(\partial\Omega)\times H^{\frac{3}{2}}(\partial\Omega)\rightarrow H^4(\Omega), \mathcal{P}_q(f,g)=u$
 for the problem
\begin{align*}
\left\{
\begin{aligned}
\Delta^2u-qu&=0\quad \mbox{in }\Omega,\\
	u&=f\quad \mbox{on }\partial\Omega,\\
	\Delta u&=g \quad\mbox{on }\partial\Omega.
\end{aligned}
\right.
\end{align*}
In addition, for any source function $F\in L^2(\Omega)$, we define the Green's operator $\mathcal{G}_q:L^2(\Omega)\rightarrow H^4(\Omega),\mathcal{G}_q(F)=v$ for the problem
\begin{align*}
\left\{
\begin{aligned}
\Delta^2v-qv&=-F\quad \mbox{in }\Omega,\\
	v&=0\quad \mbox{on }\partial\Omega,\\
	\Delta v&=0 \quad\mbox{on }\partial\Omega,
\end{aligned}
\right.
\end{align*}
\end{defi}
When the potential function is sufficiently small such that $\|q\|_{L^\infty}<\eta $ and $k^4-q\in \mathcal{Q}$, the DtN map for the problem (\ref{f}) can be presented with the above definitions as $$\mathcal{N}_q(f,g)=(\partial_\nu(\mathcal{P}_{k^4-q}(f,g))|_{\partial \Omega}, \partial_\nu(\Delta( \mathcal{P}_{k^4-q}(f,g)))|_{\partial \Omega}).$$
If we further define $\mathcal{M}_q$ to be the multiplication operator by $q\in L^\infty(\Omega)$, the linearized DtN map defined in (\ref{m4}) can be reformulated as
$$\mathcal{N}_q^1(f,g)=(\partial_\nu(\mathcal{G}_{k^4}\mathcal{M}_q\mathcal{P}_{k^4}(f,g))|_{\partial\Omega},\partial_\nu(\Delta( \mathcal{G}_{k^4}\mathcal{M}_q\mathcal{P}_{k^4}(f,g)))|_{\partial\Omega}).$$

The following theorem states the relation between the nonlinear operator $\mathcal{N}_q$ and the linearized DtN operator $\mathcal{N}_q^1$ if the potential function $q$ is assumed to be sufficiently small.
\begin{theo}
Given any $k > 0$ such that $k^4\in \mathcal{Q}$, and let $\eta$ be a sufficiently small constant satisfying $k^4-q\in \mathcal{Q}$ for any potential function $q\in L^\infty(\Omega)$ with $\|q\|_{L^\infty(\Omega)}<\eta$. Then there holds
\begin{align}
	\|\mathcal{N}_q-\mathcal{N}_q^0-\mathcal{N}_q^1\|\leq C\|q\|_{L^\infty(\Omega)}^2,
\end{align}
where $\mathcal{N}_q^0(f,g):=(\partial_\nu(\mathcal{P}_{k^4}(f,g))|_{\partial\Omega},\partial_\nu(\Delta(\mathcal{P}_{k^4}(f,g)))|_{\partial\Omega})$ and the constant $C$ depends on $n,\Omega$ and $k$.	
\end{theo}
\begin{rema}
Theorem 2.4 shows that if $\|q\|_{L^\infty(\Omega)}$ is small enough, then the error between the observation date of the original inverse boundary problem and the observation date of the linearized inverse boundary problem will be small enough.
\end{rema}
\begin{proof}
This proof is inspired by \cite{zou2022linearized}.

For all $(f,g) \in H^{\frac{7}{2}}(\partial\Omega)\times H^{\frac{3}{2}}(\partial\Omega)$, by the assumptions $k^4 \in\mathcal{Q}$ and $k^4-q\in\mathcal{Q}$, we see that
\begin{align*}
	0&=(\Delta^2-k^4+q)\mathcal{P}_{k^4-q}(f,g)-(\Delta^2-k^4)\mathcal{P}_{k^4}(f,g)\\
	&=(\Delta^2-k^4)(\mathcal{P}_{k^4-q}(f,g)-\mathcal{P}_{k^4}(f,g))+\mathcal{M}_q\mathcal{P}_{k^4-q}(f,g),
\end{align*}
which gives $$ (\Delta^2-k^4)(\mathcal{P}_{k^4-q}(f,g)-\mathcal{P}_{k^4}(f,g))=-\mathcal{M}_q\mathcal{P}_{k^4-q}(f,g) \quad \mbox{in}\ \Omega.$$	
Since $\mathcal{P}_{k^4-q}(f,g)-\mathcal{P}_{k^4}(f,g) =0$ on $\partial\Omega$, $\Delta(\mathcal{P}_{k^4-q}(f,g)-\mathcal{P}_{k^4}(f,g)) =0$ on $\partial\Omega$, from the definition of $\mathcal{G}_{k^4}$ we see that
\begin{align}
	\mathcal{P}_{k^4-q}(f,g)-\mathcal{P}_{k^4}(f,g)=\mathcal{G}_{k^4}\mathcal{M}_q\mathcal{P}_{k^4-q}(f,g) \quad \mbox{in}\ \Omega.
\label{h1}\end{align}
 Hence we know that
 \begin{align}
 	\begin{split}
 		&(\mathcal{N}_q-\mathcal{N}_q^0-\mathcal{N}_q^1)(f,g)\\
 		=&(\partial_\nu(\mathcal{P}_{k^4-q}(f,g)-\mathcal{P}_{k^4}(f,g)-\mathcal{G}_{k^4}\mathcal{M}_q\mathcal{P}_{k^4}(f,g))|_{\partial \Omega}, \\
 		&\partial_\nu(\Delta( \mathcal{P}_{k^4-q}(f,g))-\Delta(\mathcal{P}_{k^4}(f,g))-\Delta(\mathcal{G}_{k^4}\mathcal{M}_q\mathcal{P}_{k^4}(f,g)))|_{\partial \Omega})\\
 		=&(\partial_\nu(\mathcal{G}_{k^4}\mathcal{M}_q\mathcal{P}_{k^4-q}(f,g)-\mathcal{G}_{k^4}\mathcal{M}_q\mathcal{P}_{k^4}(f,g))|_{\partial \Omega},\\
 		&\partial_\nu(\Delta( \mathcal{G}_{k^4}\mathcal{M}_q\mathcal{P}_{k^4-q}(f,g))-\Delta (\mathcal{G}_{k^4}\mathcal{M}_q\mathcal{P}_{k^4}(f,g)))|_{\partial \Omega})\\
 		=&(\partial_\nu(\mathcal{G}_{k^4}\mathcal{M}_q(\mathcal{P}_{k^4-q}(f,g)-\mathcal{P}_{k^4}(f,g)))|_{\partial \Omega},\\
 		&\partial_\nu(\Delta( \mathcal{G}_{k^4}\mathcal{M}_q(\mathcal{P}_{k^4-q}(f,g)-\mathcal{P}_{k^4}(f,g))))|_{\partial \Omega})\\
 		:=&(\partial_\nu w|_{\partial\Omega},\partial_\nu (\Delta w)|_{\partial\Omega})
 	\end{split}
 \label{h7}\end{align}
by writing $w:=\mathcal{G}_{k^4}\mathcal{M}_q(\mathcal{P}_{k^4-q}(f,g)-\mathcal{P}_{k^4}(f,g)).$
	
When the smallness assumption
$$\|q\|_{L^\infty(\Omega)}<\frac{1}{2}	\|\mathcal{G}_{k^4}\|^{-1}_{\mathcal{L}(L^2(\Omega),H^4(\Omega))}$$holds, $(I-\mathcal{G}_{k^4}\mathcal{M}_q)$ is invertible in $H^4(\Omega)$. The equation (\ref{h1}) turns out that
\begin{align}
\mathcal{P}_{k^4-q}(f,g)-\mathcal{P}_{k^4}(f,g)=(I-\mathcal{G}_{k^4}\mathcal{M}_q)^{-1}\mathcal{G}_{k^4}\mathcal{M}_q\mathcal{P}_{k^4}(f,g)	\quad \mbox{in}\ \Omega.
\label{h2}\end{align}	
From (\ref{h2}), we see that	
\begin{align}
\begin{split}
&\|\mathcal{P}_{k^4-q}(f,g)-\mathcal{P}_{k^4}(f,g)\|	_{H^4(\Omega)}\\
\leq & 2\|\mathcal{G}_{k^4}\mathcal{M}_q\mathcal{P}_{k^4}(f,g)\|_{H^4(\Omega)}\\
\leq & 2\|\mathcal{G}_{k^4}\|_{\mathcal{L}(L^2(\Omega),H^4(\Omega))}\|q\|_{L^\infty(\Omega)}\| \mathcal{P}_{k^4}\|_{\mathcal{L}(H^{\frac{7}{2}}(\partial\Omega)\times H^{\frac{3}{2}}(\partial\Omega), H^4(\Omega))}\|(f,g)\|_{H^{\frac{7}{2},\frac{3}{2}}(\partial\Omega)}\\
\leq & C\|q\|_{L^\infty(\Omega)}\|(f,g)\|_{H^{\frac{7}{2},\frac{3}{2}}(\partial\Omega)}.
\end{split}
\label{h3}\end{align}	
Therefore, by using (\ref{h3}), we have the $H^4$ estimate for $w$,
\begin{align}
	\begin{split}
		\|w\|_{H^4(\Omega)}\leq & \|\mathcal{G}_{k^4}\|_{\mathcal{L}(L^2(\Omega),H^4(\Omega))}\|q\|_{L^\infty(\Omega)}\|\mathcal{P}_{k^4-q}(f,g)-\mathcal{P}_{k^4}(f,g)\|	_{L^2(\Omega)}\\
		\leq &C\|q\|_{L^\infty(\Omega)}^2\|(f,g)\|_{H^{\frac{7}{2},\frac{3}{2}}(\partial\Omega)}.
	\end{split}
\label{h4}\end{align}	
By the trace theorem, we get
\begin{align}
	\|\partial_\nu w\|_{H^\frac{5}{2}(\partial\Omega)}\leq C\|w\|_{H^4(\Omega)}
\label{h5}\end{align}	
and
\begin{align}
		\|\partial_\nu (\Delta w)\|_{H^\frac{1}{2}(\partial\Omega)}\leq C\|w\|_{H^4(\Omega)}.
\label{h6}
\end{align}

Combining (\ref{h7}), (\ref{h4})-(\ref{h6}), by setting $\eta=\min\{\frac{1}{2}, \frac{1}{2}	\|\mathcal{G}_{k^4}\|^{-1}_{\mathcal{L}(L^2(\Omega),H^4(\Omega))}\}$, we conclude the proof.

\end{proof}

\subsection{Proof of Theorem \ref{th1}}
To prove Theorem \ref{th1}, we first derive two lemmas.
\begin{lemm}
Under the assumptions in Theorem \ref{th1},
\begin{align*}
	|\mathcal{F}q(r\omega)|\leq Ck^7e^{Ca}\|\mathcal{N}_q^1\|
\end{align*}
holds for $k\geq 1, r\geq 0, \omega \in \mathbb{R}^n$, with $|\omega|=1$ and $a>0$ with $k^2+a^2>r^2/4$, where $C>0$ depends only on $n,s,M,\Omega$.
\label{l1}\end{lemm}
\begin{proof}
The proof is based on the complex exponential solutions suggested by Calder\'on	\cite{calderon2006inverse} and Faddeev \cite{faddeev1966increasing} . Let us denote $\zeta_l=\eta_l+\text{i}\xi_l, l=1,2$. We can choose $\omega^{'}, \omega^{''}\in \mathbb{R}^n$ satisfying$$\omega\cdot\omega^{'}=\omega\cdot\omega^{''}=\omega^{'}\cdot\omega^{''}=0 \quad \mbox{and} \quad| \omega^{'}|=|\omega^{''}|=1.$$
Now we set
$$\eta_1=-\frac{r}{2}\omega+\sqrt{k^2+a^2-\frac{r^2}{4}}\omega^{'},\quad\xi_1=a\omega^{''},$$
$$\eta_2=-\frac{r}{2}\omega-\sqrt{k^2+a^2-\frac{r^2}{4}}\omega^{'},\quad\xi_2=-a\omega^{''},$$
provided$$k^2+a^2\geq\frac{r^2}{4}.$$
It is clear that$$\zeta_l\cdot\zeta_l=k^2,\quad\zeta_1+\zeta_2=-r\omega,\quad|\zeta_l|^2=k^2+2a^2, \quad\mbox{for } j=1,2.$$
Then we consider the following solutions
$$u_0(x)=e^{\text{i}\zeta_1\cdot x},\quad v(x)=e^{\text{i}\zeta_2\cdot x}.$$
From the equality (\ref{e1}) we derive
\begin{align}
	\int_{\partial\Omega}\partial_\nu(\Delta u_1)v\,dS+\int_{\partial\Omega}\partial_\nu u_1(\Delta v)\,dS=-\mathcal{F}q(r\omega).
\label{e2}\end{align}
Here $\mathcal{F}(\cdot)$ denotes the Fourier transform.

We now estimate the left hand side of (\ref{e2}). To do so, we observe that (see also \cite{choudhury2015stability})
\begin{align*}
&\left|\int_{\partial\Omega}\partial_\nu(\Delta u_1)v\,dS+\int_{\partial\Omega}\partial_\nu u_1(\Delta v)\,dS\right|\\\leq&\int_{\partial\Omega}|\partial_\nu(\Delta u_1)v|\,dS+\int_{\partial\Omega}|\partial_\nu u_1(\Delta v)|\,dS\\
\leq&\|\partial_\nu(\Delta u_1)\|_{L^2(\partial\Omega)}\|v\|_{L^2(\partial\Omega)}+\|\partial_\nu u_1\|_{L^2(\partial\Omega)}\|\Delta v\|_{L^2(\partial\Omega)}\\
\leq&C(\|\partial_\nu(\Delta u_1)\|_{L^2(\partial\Omega)}\|v\|_{H^1(\Omega)}+\|\partial_\nu u_1\|_{L^2(\partial\Omega)}\|\Delta v\|_{H^1(\Omega)})\\
\leq&C(\|\partial_\nu(\Delta u_1)\|_{L^2(\partial\Omega)}+\|\partial_\nu u_1\|_{L^2(\partial\Omega)})(\|v\|_{H^1(\Omega)}+\|\Delta v\|_{H^1(\Omega)})\\
\leq&C\|\left(\partial_\nu(\Delta u_1),\partial_\nu u_1\right)\|_{H^{\frac{1}{2},\frac{5}{2}}(\partial\Omega)}(\|v\|_{H^1(\Omega)}+\|\Delta v\|_{H^1(\Omega)})\\
=&C\|\mathcal{N}_q^1(f,g)\|_{H^{\frac{5}{2},\frac{1}{2}}(\partial\Omega)}(\|v\|_{H^1(\Omega)}+\|\Delta v\|_{H^1(\Omega)})\\
\leq&C\|\mathcal{N}_q^1\|\|(f,g)\|_{H^{\frac{7}{2},\frac{3}{2}}(\partial\Omega)}(\|v\|_{H^1(\Omega)}+\|\Delta v\|_{H^1(\Omega)}).
\end{align*}
By the trace theorem, it holds
\begin{align}
|\mathcal{F}q(r\omega)|=&\left|\int_{\partial\Omega}\partial_\nu(\Delta u_1)v\,dS+\int_{\partial\Omega}\partial_\nu u_1(\Delta v)\,dS\right|\nonumber\\
\leq&C\|\mathcal{N}_q^1\|(\|u_0\|_{H^4(\Omega)}+\|\Delta u_0\|_{H^2(\Omega)})(\|v\|_{H^1(\Omega)}+\|\Delta v\|_{H^1(\Omega)}).
\label{e3}\end{align}
Therefore, we shall now have to estimate the norms of $u_0, v$ and their derivatives that appear in the above expression. Since $\Omega$ is a bounded domain, we assume that $\Omega\in B(0,R_0)$ for some fixed $R_0>0$. Then $|e^{\text{i}x\cdot\zeta_l}|=e^{-x\cdot\xi_l}\leq e^{aR_0}$, since $|\xi_l|=a$ for $l=1,2$. We can estimate
\begin{align*}
&\|v\|_{H^1(\Omega)}=\left(\int_\Omega\left(1+|\zeta_2|^2\right)|e^{\text{i}x\cdot\zeta_2}|^2\,dx\right)^{\frac{1}{2}}\leq C\left(1+k^2+2a^2\right)^{\frac{1}{2}}e^{aR_0}\leq Cke^{Ca},\\
&\|\Delta v\|_{H^1(\Omega)}=\|k^2v\|_{H^1(\Omega)}\leq Ck^3e^{Ca},\\
&\|\Delta u_0\|_{H^2(\Omega)}=k^2\|u_0\|_{H^2(\Omega)}\leq  Ck^2\left[1+k^2+2a^2+(k^2+2a^2)^2\right]^{\frac{1}{2}}e^{aR_0}\leq Ck^4e^{Ca},\\
&\|u_0\|_{H^4(\Omega)}\leq  Ck^4e^{Ca}.
\end{align*}
Using these estimates in (\ref{e3}), we have
\begin{align*}
|\mathcal{F}q(r\omega)|&\leq C\|\mathcal{N}_q^1\|(k^4e^{Ca}+k^4e^{Ca})(ke^{Ca}+k^3e^{Ca})\\
&\leq Ck^7e^{Ca}\|\mathcal{N}_q^1\|,
\end{align*}
which completes the proof.

\end{proof}

The following lemma is a corollary of Lemma \ref{l1}.
\begin{lemm}
Suppose that the assumptions in Theorem \ref{th1} hold. Let $R>0$. Then for $k\geq 1, r\geq 0 \mbox{ and } \omega \in \mathbb{R}^n$, with $|\omega|=1$, the following estimates hold true: if $0\leq r \leq k+R $ then
\begin{align}
|\mathcal{F}q(r\omega)|	\leq Ck^7e^{CR}\|\mathcal{N}_q^1\|;
\label{i1}\end{align}
if $r\geq  k+R$ then
\begin{align}
|\mathcal{F}q(r\omega)|	\leq Ck^7e^{Cr}\|\mathcal{N}_q^1\|.
\label{i2}\end{align}
\end{lemm}
\begin{proof}
We can prove the lemma immediately by taking $a=R$ when $0\leq r \leq k+R $ , and taking $a=r$ when $ r \geq k+R $ in Lemma \ref{l1}.
\end{proof}
Now we prove our main theorem.
\begin{proof}[Proof of Theorem 2.1]
The proof is inspired by \cite{iixsakov2014increasing}, where the Schr{\"o}dinger equation have been considered. By writing in polar coordinates, we can obtain
\begin{align}
\|q\|^2_{H^{-s}(\mathbb{R}^n)}=C&\int_0^\infty\int_{|\omega|=1}|\mathcal{F}q(r\omega)|^2(1+r^2)^{-s}r^{n-1}\,d\omega dr \nonumber\\
=C&\left(\int_0^{k+R}\int_{|\omega|=1}|\mathcal{F}q(r\omega)|^2(1+r^2)^{-s}r^{n-1}\,d\omega dr\right.\nonumber\\
&+\int_{k+R}^T\int_{|\omega|=1}|\mathcal{F}q(r\omega)|^2(1+r^2)^{-s}r^{n-1}\,d\omega dr\nonumber\\
&\left.+\int_T^\infty\int_{|\omega|=1}|\mathcal{F}q(r\omega)|^2(1+r^2)^{-s}r^{n-1}\,d\omega dr\right)\nonumber\\
=:C&(I_1+I_2+I_3),
\label{qq}\end{align}
where $R>0$ is a positive constant and $T\geq k+R$	is parameter which will be chosen later.

Our task now is to estimate each integral separately. We begin with $I_3$. Since $|\mathcal{F}q(r\omega)|\leq C\|q\|_{L^2(\Omega)}, q\in H^s(\Omega)$ and $s>n/2$, we get
\begin{align}
	I_3&\leq C\int_T^\infty \|q\|_{L^2(\Omega)}^2(1+r^2)^{-s}r^{n-1}\, dr\leq CT^{-m}\|q\|_{L^2(\Omega)}^2\nonumber\\
	&\leq CT^{-m}\left(\varepsilon \|q\|_{H^{-s}(\Omega)}^2+\frac{1}{\varepsilon}\|q\|_{H^s(\Omega)}^2\right)\nonumber\\
	&\leq CT^{-m}\left(\varepsilon\|q\|_{H^{-s}(\mathbb{R}^n)}^2+\frac{M}{\varepsilon}\right)
\label{i3}\end{align}
for $\varepsilon>0$, where $m:=2s-n$.

On the other hand, by estimate (\ref{i1}), we can obtain
\begin{align}
	I_1&\leq C\int_0^{k+R}k^{14}e^{CR}\|\mathcal{N}_q^1\|^2(1+r^2)^{-s}r^{n-1}\,dr\nonumber\\
	&\leq Ck^{14}e^{CR}\|\mathcal{N}_q^1\|^2\int_0^\infty(1+r^2)^{-s}r^{n-1}\,dr\nonumber\\
	&=Ck^{14}e^{CR}\|\mathcal{N}_q^1\|^2.
\label{i4}\end{align}
In the same way, using estimate (\ref{i2}), we have	
\begin{align}
	I_2&\leq C\int_{k+R}^Tk^{14}e^{Cr}\|\mathcal{N}_q^1\|^2(1+r^2)^{-s}r^{n-1}\,dr\nonumber\\
	&\leq Ck^{14}\|\mathcal{N}_q^1\|^2\int_{k+R}^T e^{Cr}(1+r^2)^{-s}r^{n-1}\,dr\nonumber\\
	&\leq Ck^{14}\|\mathcal{N}_q^1\|^2e^{CT}\int_{k+R}^T(1+r^2)^{-s}r^{n-1}\,dr\nonumber\\
	&\leq Ck^{14}\|\mathcal{N}_q^1\|^2e^{CT}\int_0^\infty (1+r^2)^{-s}r^{n-1}\,dr\nonumber\\	
	&\leq Ck^{14}e^{CT}\|\mathcal{N}_q^1\|^2,
\label{i5}\end{align}
where $s>n/2$. Combining (\ref{i3})-(\ref{i5}) gives
\begin{align}
\|q\|^2_{H^{-s}(\mathbb{R}^n)}&\leq C(I_1+I_2+I_3)\nonumber\\
&\leq C	k^{14}e^{CR}\|\mathcal{N}_q^1\|^2+ Ck^{14}e^{CT}\|\mathcal{N}_q^1\|^2+ CT^{-m}\left(\varepsilon\|q\|_{H^{-s}(\mathbb{R}^n)}^2+\frac{M}{\varepsilon}\right)\nonumber\\
&= CT^{-m}\varepsilon\|q\|_{H^{-s}(\mathbb{R}^n)}^2+C	k^{14}e^{CR}\|\mathcal{N}_q^1\|^2+ Ck^{14}e^{CT}\|\mathcal{N}_q^1\|^2+\frac{CT^{-m}}{\varepsilon}.
\label{l2}\end{align}

To continue, we consider two cases:
$$(i)k+R\leq p\log\frac{1}{A}$$
and
$$(ii)k+R\geq p\log\frac{1}{A},$$
where $A=\|\mathcal{N}_q^1\|^2$ and $p>0$ is a constant which can be determined later.

For the case (i), taking $$\varepsilon=\frac{T^m}{2C}$$
and $R>0$,
we deduce that
\begin{align}
\|q\|^2_{H^{-s}(\mathbb{R}^n)}\leq C k^{14}A+ Ck^{14}e^{CT}A+CT^{-2m}
\label{q1}\end{align}
for any $T\geq k+R$ by (\ref{l2}).
Now we choose $T=p\log\frac{1}{A}$, which is greater than or equal to $k+R$ by the condition (i). Our current aim is to show that there exists $C_1\geq 0$ such that
\begin{align}
k^{14}e^{CT}A+T^{-2m}\leq 2C_1\left(k+\log\frac{1}{A}\right)^{-2m}.
\label{q2}\end{align}
Substituting (\ref{q2}) into (\ref{q1}) clearly implies (\ref{mm}). Now to derive (\ref{q2}), it is enough to prove that
\begin{align}
k^{14}e^{CT}A\leq C_1\left(k+\log\frac{1}{A}\right)^{-2m}	
\label{q3}\end{align}
and
\begin{align}
T^{-2m}\leq C_1\left(k+\log\frac{1}{A}\right)^{-2m}.
\label{q4}\end{align}
We note that (\ref{q4}) is equivalent to
\begin{align}
	C_1^{-1/2m}\left(k+\log\frac{1}{A}\right)\leq p\log\frac{1}{A} .
\label{q5}	\end{align}
Since we have $$k+\log\frac{1}{A}\leq (k+R)+\log\frac{1}{A}\leq (p+1)\log\frac{1}{A}$$
by (i), condition (\ref{q5}) holds whenever
\begin{align}
	C_1^{-1/2m}\leq \frac{p}{p+1}.	
\label{aa}\end{align}
Now we turn to (\ref{q3}). It is clear that (\ref{q3}) is equivalent to
\begin{align}
	14\log k+(Cp-1)\log \frac{1}{A} +2m\log \left(k+\log\frac{1}{A}\right)\leq \log C_1
\label{q6}\end{align}
since $T=p\log\frac{1}{A}$. Using (i), we can obtain
$$\log k\leq \log p+\log(\log\frac{1}{A}),$$
and
$$\log \left(k+\log\frac{1}{A}\right)\leq \log \left(p\log\frac{1}{A}+\log\frac{1}{A}\right)=\log(p+1)+\log(\log\frac{1}{A}).$$
Hence (\ref{q6}) is verified if we can show that
\begin{align}
14\log p+2m\log(p+1)+(Cp-1)\log \frac{1}{A} +(2m+14)\log(\log\frac{1}{A})\leq \log C_1.
\label{q7}\end{align}
Choosing
\begin{align}
	p\leq \frac{1}{2C},
\label{q8}\end{align}
we can bound the left-hand side of (\ref{q7}) by
\begin{align*}
	&(\mbox{LHS of }(\ref{q7}))\\
	&\leq 14\log\frac{1}{2C}+2m\log(\frac{1}{2C}+1)-\frac{1}{2}\log\frac{1}{A}+(2m+14)\log(\log\frac{1}{A})\\
	&\leq 14\log\frac{1}{2C}+2m\log(\frac{1}{2C}+1)+\max_{z\geq 2}\left(-\frac{1}{2}z+(2m+14)\log z\right)\\
	&=14\log\frac{1}{2C}+2m\log(\frac{1}{2C}+1)+2(m+7)(\log(4m+28)-1).
\end{align*}
Therefore, condition (\ref{q7}) (i.e.(\ref{q3})) is satisfied provided
\begin{align}
14\log\frac{1}{2C}+2m\log(\frac{1}{2C}+1)+2(m+7)(\log(4m+28)-1)\leq \log{C_1}.	
\label{q9}\end{align}

Next we consider case (ii). We choose $T=k+R$ and observe that the term $I_2$ in (\ref{qq}) does not appear in this case. Hence, instead of (\ref{l2}), we have
$$\|q\|^2_{H^{-s}(\mathbb{R}^n)}\leq CT^{-m}\varepsilon\|q\|_{H^{-s}(\mathbb{R}^n)}^2+C	k^{14}e^{CR}\|\mathcal{N}_q^1\|^2+\frac{CT^{-m}}{\varepsilon}.$$
Set $\varepsilon=T^m/2C$ and $R>0$, we obtain
\begin{align*}
\|q\|^2_{H^{-s}(\mathbb{R}^n)}&\leq C	k^{14}e^{CR}\|\mathcal{N}_q^1\|^2+CT^{-2m}\\
&\leq 	Ck^{14}\|\mathcal{N}_q^1\|^2+C(k+R)^{-2m}
\end{align*}
which implies the desired estimate (\ref{mm}) since from condition (ii) we have
$$k+R\geq \frac{k}{2}+\frac{k+R}{2}\geq \frac{k}{2}+\frac{p}{2}\log\frac{1}{A}\geq \frac{\min\{p,1\}}{2}\left(k+\log\frac{1}{A}\right).$$

As the last step, we choose appropriate $R$,  $p$ and $C_1$ to complete the proof. We first pick an arbitrary positive constant $R$ and then choose $p$ small enough satisfying (\ref{q8}). Finally, we take $C_1$ large enough satisfying (\ref{aa}) and (\ref{q9}).

\end{proof}

\section{Linearized inverse biharmonic potential problem at the large wavenumber with attenuation}
Furthermore we provide some extended discussion on the linearized inverse biharmonic potential problem with attenuation.

In current section, we investigate the stability estimate recovering the potential function $q=q(x)$ below
\begin{align}
\left\{
\begin{aligned}
\Delta^2u-(k^2+\text{i}kb)^2u+qu&=0\quad \mbox{in }\Omega,\\
	u&=f\quad \mbox{on }\partial\Omega,\\
	\Delta u&=g \quad\mbox{on }\partial\Omega,
\end{aligned}
\right.
\label{a1}\end{align}
from the knowledge of linearized DtN map for a large wavenumber $k$. The constant $b>0$ is the attenuation constant. For the sake of simplicity, we use the notations in section 2, for instance, that $u$ represents the solution of the biharmonic equation with attenuation. We suppose that there exists a unique solution to (\ref{a1}).

Referring to section 2, we let $u_0$, $u_1$ solve the following subproblems

\begin{align}
\left\{
\begin{aligned}
\Delta^2u_0-(k^2+\text{i}kb)^2u_0&=0\quad \mbox{in }\Omega,\\
	u_0&=f\quad \mbox{on }\partial\Omega,\\
	\Delta u_0&=g \quad\mbox{on }\partial\Omega,
\end{aligned}
\right.
\label{a2}\end{align}
\begin{align}
\left\{
\begin{aligned}
	\Delta^2u_1-(k^2+\text{i}kb)^2u_1&=-qu_0&&\mbox{in } \Omega,\\
	u_1&=0 &&\mbox{on }\partial\Omega,\\
	\Delta u_1&=0 &&\mbox{on }\partial\Omega,	
\end{aligned}
\right.
\label{a3}\end{align}
then the solution $u$ of the original problem (\ref{a1}) is
$$u=u_0+u_1+\cdots$$where the remaining ``$\cdots$" are ``higher" order terms. Similarly the linearized DtN map $\mathcal{N}_q^1$  is defined as
\begin{align}
	\mathcal{N}_q^1:(f,g)\longmapsto(\partial_\nu u_1,\partial_\nu (\Delta u_1)) \quad \mbox{on }\partial\Omega.
	\label{a4}
\end{align}

We consider the stability for the determination of $q$ from the  linearized DtN map $\mathcal{N}_q^1$ defined by (\ref{a4}).

Multiplying both sides of the subproblem (\ref{a3}) with a test function $v$ satisfying
$$\Delta^2v-(k^2+\text{i}kb)^2v=0\quad\mbox{in }\Omega,$$
we obtain the equality
\begin{align}
	\int_{\partial\Omega}\partial_\nu(\Delta u_1)v\,dS+\int_{\partial\Omega}\partial_\nu u_1(\Delta v)\,dS=-\int_\Omega qu_0v\,dx
\label{a5}\end{align}
by the Green's formula.

Since $\Omega$ is a bounded domain, we assume that $\Omega \in B(0,R_0)$ for some fixed $R_0>0$. Then we present the main stability estimate in this section.
\begin{theo}
Let $n\geq 3$, $s>n/2$. Assume the potential function $q\in H^s(\Omega)$ with supp$q\subset\Omega, \| q\|_{H^{s}(\Omega)}\leq M$. Then the following estimate holds true
\begin{align}\|q\|_{H^{-s}(\mathbb{R}^n)}\leq C (k+\sqrt{kb})^7e^{C\sqrt{kb}}\|\mathcal{N}_q^1\|+C(k+\sqrt{kb}+\log\frac{1}{\|\mathcal{N}_q^1\|^2})^{-(2s-n)}
\label{m}\end{align}	
for the linearized system (\ref{a2})-(\ref{a3}) with  $\|\mathcal{N}_q^1\|\leq 1/e$ and the constant $C$ depends only on $n,s,\Omega, M$ and \mbox{supp}q.

\label{th2}	
\end{theo}

Similar to the proof of Theorem \ref{th1}, we again derive two lemmas.

\begin{lemm}
Under the assumptions in Theorem \ref{th2},
\begin{align*}
	|\mathcal{F}q(r\omega)|\leq C(k+\sqrt{kb})^7e^{C(\sqrt{kb}+a)}\|\mathcal{N}_q^1\|\end{align*}
holds for $k\geq 1, r\geq 0, \omega \in \mathbb{R}^n$, with $|\omega|=1$ and $a>0$ with $k^2+a^2>r^2/4$, where $C>0$ depends only on $n,s,M,\Omega$.
\label{l21}\end{lemm}

\begin{proof}
We again use the exponential solutions. Let us denote $\zeta_l=\eta_l+\text{i}\xi_l, l=1,2$. We can choose $\omega^{'}, \omega^{''}\in \mathbb{R}^n$ satisfying$$\omega\cdot\omega^{'}=\omega\cdot\omega^{''}=\omega^{'}\cdot\omega^{''}=0 \quad \mbox{and} \quad| \omega^{'}|=|\omega^{''}|=1.$$
Now we set
$$\eta_1=-\frac{r}{2}\omega+\sqrt{k^2+a^2-\frac{r^2}{4}+\text{i}kb}\omega^{'},\quad\xi_1=a\omega^{''},$$
$$\eta_2=-\frac{r}{2}\omega-\sqrt{k^2+a^2-\frac{r^2}{4}+\text{i}kb}\omega^{'},\quad\xi_2=-a\omega^{''},$$
provided$$k^2+a^2\geq\frac{r^2}{4}.$$
It is clear that$$\zeta_l\cdot\zeta_l=k^2+\text{i}kb,\quad\zeta_1+\zeta_2=-r\omega, \quad\mbox{for } l=1,2.$$
Then we consider the following solutions
$$u_0(x)=e^{\text{i}\zeta_1\cdot x},\quad v(x)=e^{\text{i}\zeta_2\cdot x}.$$
From the equality (\ref{a5}) we derive
\begin{align*}
	\int_{\partial\Omega}\partial_\nu(\Delta u_1)v\,dS+\int_{\partial\Omega}\partial_\nu u_1(\Delta v)\,dS=-\mathcal{F}q(r\omega).
\end{align*}
Here $\mathcal{F}(\cdot)$ denotes the Fourier transform.

We write $\sqrt{k^2+a^2-\frac{r^2}{4}+\text{i}kb}=X+\text{i}Y$, $X>0$. Squaring the both sides yields $$k^2+a^2-\frac{r^2}{4}=X^2-Y^2, \quad kb=2XY.$$ Substituting $X=kb/2Y$
into the first equation and solving the resulting (quadratic) equation we obtain
\begin{align*}
Y&=\frac{\sqrt{-(4k^2+4a^2-r^2)+\sqrt{(4k^2+4a^2-r^2)^2+16k^2b^2}}}{2\sqrt{2}}\\
&\leq \sqrt{\frac{kb}{2}},
\end{align*}
where we used the inequality $\sqrt{A^2+B^2}\leq A+B$ for any positive $A, B$. It follows that
$$|e^{\text{i}x\cdot \zeta_1}|=e^{-x\cdot(Y\omega^{'}+a\omega^{''})}\leq e^{R_0(Y+a)}\leq e^{R_0(\sqrt{\frac{kb}{2}}+a)}$$
and
$$|e^{\text{i}x\cdot \zeta_2}|=e^{x\cdot(Y\omega^{'}+a\omega^{''})}\leq e^{R_0(Y+a)}\leq e^{R_0(\sqrt{\frac{kb}{2}}+a)}.$$
Similarly, we have
$$|\zeta_l |^2=\frac{r^2}{4}+a^2+\sqrt{(k^2+a^2-\frac{r^2}{4})^2+k^2b^2}\leq k^2+kb+2a^2, \quad l=1,2.$$
Using these estimates, we then derive that
\begin{align*}
&\|v\|_{H^1(\Omega)}\leq C\left(1+k^2+kb+2a^2\right)^{\frac{1}{2}}e^{R_0(\sqrt{\frac{kb}{2}}+a)}\leq C(k+\sqrt{kb})e^{C(\sqrt{kb}+a)},\\
&\|\Delta v\|_{H^1(\Omega)}=\|(k^2+\text{i}kb)v\|_{H^1(\Omega)}\leq C(k+\sqrt{kb})^3e^{C(\sqrt{kb}+a)},\\
&\|\Delta u_0\|_{H^2(\Omega)}=\|(k^2+\text{i}kb)u_0\|_{H^2(\Omega)}\leq C(k+\sqrt{kb})^4e^{C(\sqrt{kb}+a)},\\
&\|u_0\|_{H^4(\Omega)}\leq  C(k+\sqrt{kb})^4e^{C(\sqrt{kb}+a)}.
\end{align*}

Proceeding similarly as the case in section 2, we have
\begin{align*}
|\mathcal{F}q(r\omega)|&\leq C\|\mathcal{N}_q^1\|\left((k+\sqrt{kb})^4+(k+\sqrt{kb})^4\right)\left((k+\sqrt{kb})+(k+\sqrt{kb})^3\right)e^{C(\sqrt{kb}+a)}\\
&\leq C(k+\sqrt{kb})^7e^{C(\sqrt{kb}+a)}\|\mathcal{N}_q^1\|
\end{align*}
from (\ref{e3}). This concludes the proof.

\end{proof}

The following lemma is a corollary of Lemma \ref{l21}.
\begin{lemm}
Suppose that the assumptions in Theorem \ref{th2} hold. Let $R>0$. Then for $k\geq 1, r\geq 0 \mbox{ and } \omega \in \mathbb{R}^n$, with $|\omega|=1$, the following estimates hold true: if $0\leq r \leq k+\sqrt{kb}+R $ then
\begin{align}
|\mathcal{F}q(r\omega)|	\leq C(k+\sqrt{kb})^7e^{C(\sqrt{kb}+R)}\|\mathcal{N}_q^1\|;
\label{z1}\end{align}
if $r\geq  k+\sqrt{kb}+R$ then
\begin{align}
|\mathcal{F}q(r\omega)|	\leq C(k+\sqrt{kb})^7e^{C(\sqrt{kb}+r)}\|\mathcal{N}_q^1\|.
\label{z2}\end{align}
\end{lemm}
\begin{proof}
We can prove Lemma 3.3 immediately by taking $a=\sqrt{kb}+R$ when $0\leq r \leq k+\sqrt{kb}+R $ , and taking $a=r$ when $ r \geq k+\sqrt{kb}+R $ in Lemma \ref{l21}.
\end{proof}

Now we prove Theorem \ref{th2}.
\begin{proof}[Proof of Theorem 3.1]
Similarly to the proof in Theorem \ref{th1}, we have		
\begin{align}
\|q\|^2_{H^{-s}(\mathbb{R}^n)}
=C&\left(\int_0^{k+\sqrt{kb}+R}\int_{|\omega|=1}|\mathcal{F}q(r\omega)|^2(1+r^2)^{-s}r^{n-1}\,d\omega dr\right.\nonumber\\
&+\int_{k+\sqrt{kb}+R}^T\int_{|\omega|=1}|\mathcal{F}q(r\omega)|^2(1+r^2)^{-s}r^{n-1}\,d\omega dr\nonumber\\
&\left.+\int_T^\infty\int_{|\omega|=1}|\mathcal{F}q(r\omega)|^2(1+r^2)^{-s}r^{n-1}\,d\omega dr\right)\nonumber\\
=:C&(I_1+I_2+I_3),
\label{zz}\end{align}
where $R>0$ and $T\geq k+\sqrt{kb}+R$	are parameters which will be chosen later.
In addition, it holds that
\begin{align}
	I_3\leq CT^{-m}\left(\varepsilon\|q\|_{H^{-s}(\mathbb{R}^n)}^2+\frac{M}{\varepsilon}\right)
\label{z3}\end{align}
for $\varepsilon>0$, where $m:=2s-n$.

On the other hand, by estimate (\ref{z1}), we can obtain
\begin{align}
	I_1&\leq C\int_0^{k+\sqrt{kb}+R}(k+\sqrt{kb})^{14}e^{C(\sqrt{kb}+R)}\|\mathcal{N}_q^1\|^2(1+r^2)^{-s}r^{n-1}\,dr\nonumber\\
	&\leq C(k+\sqrt{kb})^{14}e^{C(\sqrt{kb}+R)}\|\mathcal{N}_q^1\|^2\int_0^\infty(1+r^2)^{-s}r^{n-1}\,dr\nonumber\\
	&=C(k+\sqrt{kb})^{14}e^{C(\sqrt{kb}+R)}\|\mathcal{N}_q^1\|^2.
\label{z4}\end{align}
In the same way, using estimate (\ref{z2}), we have	
\begin{align}
	I_2&\leq C\int_{k+\sqrt{kb}+R}^T(k+\sqrt{kb})^{14}e^{C(\sqrt{kb}+r)}\|\mathcal{N}_q^1\|^2(1+r^2)^{-s}r^{n-1}\,dr\nonumber\\
	&\leq C(k+\sqrt{kb})^{14}\|\mathcal{N}_q^1\|^2\int_{k+\sqrt{kb}+R}^Te^{C(\sqrt{kb}+r)}(1+r^2)^{-s}r^{n-1}\,dr\nonumber\\
	&\leq C(k+\sqrt{kb})^{14}e^{C(\sqrt{kb}+T)}\|\mathcal{N}_q^1\|^2\int_{k+\sqrt{kb}+R}^T(1+r^2)^{-s}r^{n-1}\,dr\nonumber\\
	&\leq C(k+\sqrt{kb})^{14}e^{C(\sqrt{kb}+T)}\|\mathcal{N}_q^1\|^2\int_0^\infty (1+r^2)^{-s}r^{n-1}\,dr\nonumber\\	
	&\leq C(k+\sqrt{kb})^{14}e^{C(\sqrt{kb}+T)}\|\mathcal{N}_q^1\|^2,
\label{z5}\end{align}
where $s>n/2$. Combining (\ref{z3})-(\ref{z5}) gives

\begin{align}
\|q\|^2_{H^{-s}(\mathbb{R}^n)}\leq &C(I_1+I_2+I_3)\nonumber\\
\leq &C(k+\sqrt{kb})^{14}e^{C(\sqrt{kb}+R)}\|\mathcal{N}_q^1\|^2+ C(k+\sqrt{kb})^{14}e^{C(\sqrt{kb}+T)}\|\mathcal{N}_q^1\|^2\nonumber\\
&+ CT^{-m}\left(\varepsilon\|q\|_{H^{-s}(\mathbb{R}^n)}^2+\frac{M}{\varepsilon}\right)\nonumber\\
=& CT^{-m}\varepsilon\|q\|_{H^{-s}(\mathbb{R}^n)}^2+C(k+\sqrt{kb})^{14}e^{C(\sqrt{kb}+R)}\|\mathcal{N}_q^1\|^2\nonumber\\
&+ C(k+\sqrt{kb})^{14}e^{C(\sqrt{kb}+T)}\|\mathcal{N}_q^1\|^2
+\frac{CT^{-m}}{\varepsilon}.
\label{z6}\end{align}

To continue, we consider two cases:
$$(i)k+\sqrt{kb}+R\leq p\log\frac{1}{A}$$
and
$$(ii)k+\sqrt{kb}+R\geq p\log\frac{1}{A},$$
where $A=\|\mathcal{N}_q^1\|^2, R>0$ and $p>0$ are constants which can be determined later.

For the case (i), taking $$\varepsilon=\frac{T^m}{2C}$$and $R>0$,
we deduce that
\begin{align}
\|q\|^2_{H^{-s}(\mathbb{R}^n)}\leq C (k+\sqrt{kb})^{14}e^{C\sqrt{kb}}A+ C(k+\sqrt{kb})^{14}e^{C(\sqrt{kb}+T)}A+CT^{-2m}
\label{z7}\end{align}
for any $T\geq k+\sqrt{kb}+R$ by (\ref{z6}).
Now we choose $T=p\log\frac{1}{A}$, which is greater than or equal to $k+\sqrt{kb}+R$ by the condition (i). Our current aim is to show that there exists $C_1\geq 0$ such that
\begin{align}
(k+\sqrt{kb})^{14}e^{C(\sqrt{kb}+T)}A+T^{-2m}\leq 2C_1\left(k+\sqrt{kb}+\log\frac{1}{A}\right)^{-2m}.
\label{z8}\end{align}
Substituting (\ref{z8}) into (\ref{z7}) clearly implies (\ref{m}). Now to derive (\ref{z8}), it is enough to prove that

\begin{align}
(k+\sqrt{kb})^{14}e^{C(\sqrt{kb}+T)}A\leq C_1\left(k+\sqrt{kb}+\log\frac{1}{A}\right)^{-2m}
\label{z9}\end{align}
and
\begin{align}
T^{-2m}\leq C_1\left(k+\sqrt{kb}+\log\frac{1}{A}\right)^{-2m}.
\label{z10}\end{align}
Remark that (\ref{z10}) is equivalent to
\begin{align}
	C_1^{-1/2m}\left(k+\sqrt{kb}+\log\frac{1}{A}\right)\leq p\log\frac{1}{A} .
\label{z11}	\end{align}
Since we have
$$k+\sqrt{kb}+\log\frac{1}{A}\leq (k+\sqrt{kb}+R)+\log\frac{1}{A}\leq (p+1)\log\frac{1}{A}$$
by (i), condition (\ref{z11}) holds whenever
\begin{align}
	C_1^{-1/2m}\leq \frac{p}{p+1}.	
\label{z16}\end{align}
Now we turn to (\ref{z9}). It is clear that (\ref{z9}) is equivalent to
\begin{align}
	14\log (k+\sqrt{kb})+C\sqrt{kb}+(Cp-1)\log \frac{1}{A} +2m\log \left(k+\sqrt{kb}+\log\frac{1}{A}\right)\leq \log C_1
\label{z12}\end{align}
since $T=p\log\frac{1}{A}$. Using (i), we can obtain
$$\sqrt{kb}\leq p\log\frac{1}{A},$$
$$\log (k+\sqrt{kb})\leq \log p+\log(\log\frac{1}{A}),$$
and
$$\log \left(k+\sqrt{kb}+\log\frac{1}{A}\right)\leq \log \left(p\log\frac{1}{A}+\log\frac{1}{A}\right)=\log(p+1)+\log(\log\frac{1}{A}).$$
Hence (\ref{z12}) is verified if we can show that
\begin{align}
14\log p+2m\log(p+1)+(2Cp-1)\log \frac{1}{A} +(2m+14)\log(\log\frac{1}{A})\leq \log C_1.
\label{z13}\end{align}
Choosing
\begin{align}
	p\leq \frac{1}{4C},
\label{z14}\end{align}
we can bound the left-hand side of (\ref{z13}) by
\begin{align*}
	&(\mbox{LHS of }(\ref{z13}))\\
	&\leq 14\log\frac{1}{4C}+2m\log(\frac{1}{4C}+1)-\frac{1}{2}\log\frac{1}{A}+(2m+14)\log(\log\frac{1}{A})\\
	&\leq 14\log\frac{1}{4C}+2m\log(\frac{1}{4C}+1)+\max_{z\geq 2}\left(-\frac{1}{2}z+(2m+14)\log z\right)\\
	&=14\log\frac{1}{4C}+2m\log(\frac{1}{4C}+1)+2(m+7)(\log(4m+28)-1).
\end{align*}
Therefore, condition (\ref{z13}) (i.e.(\ref{z9})) is satisfied provided
\begin{align}
14\log\frac{1}{4C}+2m\log(\frac{1}{4C}+1)+2(m+7)(\log(4m+28)-1)\leq \log{C_1}.	
\label{z15}\end{align}

Next we consider case (ii). We choose $T=k+\sqrt{kb}+R$ and observe that the term $I_2$ in (\ref{zz}) does not appear in this case. Hence, instead of (\ref{z6}), we have
$$\|q\|^2_{H^{-s}(\mathbb{R}^n)}\leq  CT^{-m}\varepsilon\|q\|_{H^{-s}(\mathbb{R}^n)}^2+C(k+\sqrt{kb})^{14}e^{C(\sqrt{kb}+R)}\|\mathcal{N}_q^1\|^2+\frac{CT^{-m}}{\varepsilon}.$$
Set $\varepsilon=T^m/2C$ and $R>0$, we obtain
\begin{align*}
\|q\|^2_{H^{-s}(\mathbb{R}^n)}&\leq C (k+\sqrt{kb})^{14}e^{C\sqrt{kb}}A+CT^{-2m}\\
&=C(k+\sqrt{kb})^{14}e^{C\sqrt{kb}}A+C(k+\sqrt{kb}+R)^{-2m}
\end{align*}
which implies the desired estimate (\ref{m}) since from condition (ii) we have
\begin{align*}
k+\sqrt{kb}+R&\geq \frac{k+\sqrt{kb}}{2}+\frac{k+\sqrt{kb}+R}{2}\geq \frac{k+\sqrt{kb}}{2}+\frac{p}{2}\log\frac{1}{A}\\
&\geq \frac{\min\{p,1\}}{2}\left(k+\sqrt{kb}+\log\frac{1}{A}\right).
\end{align*}	
	
As the last step, we choose appropriate $R$,  $p$ and $C_1$ to complete the proof. We first pick an arbitrary positive constant $R$ and then choose $p$ small enough satisfying (\ref{z14}). Finally, we take $C_1$ large enough satisfying (\ref{z16}) and (\ref{z15}).

\end{proof}

\section{conclusion}

The original inverse problem for the biharmonic potential is ill-posed and nonlinear. Like \cite{isakovlu2020linearized}, to study stability we prefer to focus on the most serious difficulty, ill-conditioning, and we linearized inverse problem to avoid additional difficulties with multiple local minima. We found that recovery of the potential from all boundary data at a fixed wavenumber is dramatically improving for a larger $k$.  We will discuss the problem for some nonlinear biharmonic operators in our forthcoming papers.

\end{document}